\documentclass{amsart}
\usepackage{amsmath}
\usepackage{amsfonts}
\usepackage{amssymb}
\usepackage{latexsym}
\usepackage{times}

\newtheorem{theorem}{Theorem}[section]
\newtheorem{lemma}[theorem]{Lemma}
\newtheorem{coro}[theorem]{Corollary}
\newtheorem{propo}[theorem]{Proposition}

\newcommand{\D}{{\mathbb D}}
\newcommand{\e}{{\varepsilon}}

\newcommand{\p}{\varphi}

\begin{document}

\title
[Carleson measures on function spaces]{Carleson measures, Riemann-Stieltjes and multiplication operators
on a general family of function spaces }
\author[J. Pau]
{Jordi Pau}
\address{
Jordi Pau\\
Departament de Matem\`atica Aplicada i Analisi,
Universitat de Barcelona,
08007 Barcelona,
Spain}
\email{jordi.pau@ub.edu}

\author[R. Zhao]
{Ruhan Zhao}
\address{
Ruhan Zhao\\
Department of Mathematics,
SUNY Brockport,
Brockport, NY 14420,
USA}
\email{rzhao@brockport.edu}


\thanks{The first author is
 supported by SGR grant $2009$SGR $420$ (Generalitat de
Catalunya) and DGICYT grant MTM$2011$-$27932$-$C02$-$01$
(MCyT/MEC)}

\begin{abstract}
\par Let $\mu$ be a nonnegative Borel measure on the unit disk of the complex plane.
We characterize those measures $\mu$ such that the general family
of spaces of analytic functions, $F(p,q,s)$, which contain many
classical function spaces, including the Bloch space, $BMOA$ and
the $Q_s$ spaces, are embedded boundedly or compactly into the
tent-type spaces $T^{\infty}_{p,s}(\mu)$. The results are applied
to characterize boundedness and compactness of Riemann-Stieltjes
operators and multiplication operators on $F(p,q,s)$.
\end{abstract}
\maketitle



\begin{section} {Introduction}
\end{section}

The Carleson measure was introduced by Carleson \cite{carleson}
for studying the problem of interpolation by bounded analytic
functions and for solving the famous corona problem. Later on
variations of the Carleson measure have been also introduced and
studied. It was found that Carleson type measures are usually
closely related to certain function spaces. For example, the
original Carleson measure is closely related to Hardy spaces
$H^p$. This feature makes Carleson type measures important tools
for the modern function theory and operator theory.

In this paper we study Carleson measures related to the
general analytic function space $F(p,q,s)$ on the unit disk
of the complex plane,
introduced by the second author in \cite{zhao1}.
Basically, we are going to characterize those nonnegative Borel
measures $\mu$ such that the embedding from $F(p,q,s)$
to a certain tent-type space is bounded or compact.
The results will be applied to get characterizations
for the bounded and compact Riemann-Stieltjes operators and
pointwise multiplication operators on $F(p,q,s)$.

Our work involves several spaces of analytic functions.
Let us first review these spaces.
Let $D$ be the unit disk on the complex plane.
Let $H(D)$ be the space of all analytic functions on $D$.
For $\alpha>-1$ and $p>0$, the weighted Bergman space $A^p_{\alpha}$
consists of all functions $f\in H(D)$ such that
$$
\|f\|_{A^p_{\alpha}}=\left(\int_D|f(z)|^p\,dA_{\alpha}(z)\right)^{1/p}<\infty.
$$
Here $dA_{\alpha}(z)=(\alpha+1)(1-|z|^2)^{\alpha}\,dA(z)$,
and $dA$ is the normalized area measure on $D$.
The weighted Dirichlet space $D^p_{\alpha}$ consists of all functions $f\in H(D)$
such that
$$
\|f\|_{D^p_{\alpha}}=|f(0)|+\left(\int_D|f'(z)|^p\,dA_{\alpha}(z)\right)^{1/p}<\infty.
$$
It is well-known that $D^2_1=H^2$, the Hardy space,
and for $\alpha>p-1$, $D^p_{\alpha}=A^p_{\alpha-p}$, with equivalence of norms.

Next, we recall the Bloch type spaces (also called $\alpha$-Bloch spaces)
$B^{\alpha}$. Let $\alpha>0$. The $\alpha$-Bloch space $B^{\alpha}$ consists of
all analytic functions $f$ on $D$ such that
$$
\|f\|_{B^{\alpha}}=\sup_{z\in D}|f'(z)|(1-|z|^2)^{\alpha}<\infty.
$$
The little $\alpha$-Bloch space $B^{\alpha}_0$ consists of all analytic
functions $f$ on $D$ for which
$$
\lim_{|z|\to1}|f'(z)|(1-|z|^2)^{\alpha}=0.
$$
It is known that $B^{\alpha}$ is a Banach space under the norm
$$
|\|f\||_{B^{\alpha}}=|f(0)|+\|f\|_{B^{\alpha}},
$$
and $B^{\alpha}_0$ is a closed subspace of $B^{\alpha}$.
As $\alpha=1$, $B^1=B$, the Bloch space.

Here are some related facts about $B^{\alpha}$.
When $0<\alpha<1$, $B^{\alpha}=\operatorname{Lip}_{1-\alpha}$,
the Lipschitz type space, which contains analytic functions $f$ on $D$
for which there is a constant $C>0$ such that
$$
|f(z)-f(w)|\le C|z-w|^{1-\alpha}
$$
for all $z,w\in D$. As a simple consequence, we know that, when
$0<\alpha<1$, $B^{\alpha}\subset H^{\infty}$, the space of all
bounded analytic functions $f$ on $D$ with
$$
\|f\|_{H^{\infty}}=\sup_{z\in D}|f(z)|<\infty.
$$

When $\alpha>1$, an analytic function $f\in B^{\alpha}$
if and only if
$$
\sup_{z\in D}|f(z)|(1-|z|^2)^{\alpha-1}<\infty,
$$
and the above supremum is comparable to $|\|f\||_{B^{\alpha}}$.
For all these results, see \cite{zhu2}.

Next, we introduce the spaces $F(p,q,s)$.
For a point $a\in D$, let $\p_a(z)=(a-z)/(1-\bar az)$ denote the M\"obius
transformation of $D$ that interchanges $0$ and $a$.
An easy calculation shows
$$
\p_a'(z)=-\frac{1-|a|^2}{(1-\bar az)^2}.
$$
For $0<p<\infty$, $-2<q<\infty$, $0\le s<\infty$.
The space $F(p,q,s)$ is defined as the space of all functions
$f\in H(D)$ such that
$$
\|f\|_{F(p,q,s)}^p
=\sup_{a\in D}\int_D|f'(z)|^p(1-|z|^2)^q(1-|\p_a(z)|^2)^s\,dA(z)<\infty.
$$
It is known that, for $p\ge 1$, $F(p,q,s)$ is a Banach space
under the norm
$$
|\|f\||_{F(p,q,s)}=|f(0)|+\|f\|_{F(p,q,s)}.
$$
For $0<p<1$, the space $F(p,q,s)$ is a complete metric space with the metric given by
$$d(f,g)=|\|f-g\||^p_{F(p,q,s)}.$$
In other words, it is an $F$-space, in the terminology introduced by Banach \cite{banach}.
The family of spaces $F(p,q,s)$ was introduced in \cite{zhao1}.
It contains, as special cases, many classical function spaces,
such as the analytic Besov spaces, weighted Bergman spaces,
weighted Dirichlet spaces, the $\alpha$-Bloch spaces, BMOA and the recently
introduced $Q_s$ spaces.

For convenience, we will write $q=p\alpha-2$, where $\alpha>0$.
It is known that, for any $\alpha>0$, $F(p,p\alpha-2,s)$ are subspaces
of $B^{\alpha}$, the $\alpha$-Bloch space. As $s>1$,
we actually have $F(p,p\alpha-2,s)=B^{\alpha}$.
Also, it is known that $F(p,q,s)$ contains only constant functions
if $s+q\le-1$ or $s+p\alpha\le 1$ when $q=p\alpha-2$
(see Proposition 2.12 in \cite{zhao1}).
Therefore, later on we will assume that $s+p\alpha>1$, which
guarantees that $F(p,p\alpha-2,s)$ is nontrivial.

Among these $F(p,p\alpha-2,s)$ spaces, the case $\alpha=1$ is
particularly interesting, since the spaces $F(p,p-2,s)$ are
M\"obius invariant, in the sense that for any function $f\in
F(p,p-2,s)$ and any $a\in D$, one has
$$
\|f\circ\p_a\|_{F(p,p-2,s)}=\|f\|_{F(p,p-2,s)}.
$$
As mentioned above, when $s>1$ we have that $F(p,p-2,s)=B$, the Bloch space.
As $p=2$, $F(p,p-2,s)=Q_s$, the $Q_s$ spaces,
introduced in \cite{axz}.
As $p=2$ and $s=1$, $F(p,p-2,s)=BMOA$, the space
of analytic functions of bounded mean oscillation.
See \cite{zhao1} for details of all of the above facts about $F(p,q,s)$ spaces.

The tent-type spaces used in this paper are defined as follows.
Let $I$ be an arc on the unit circle $\partial D$.
Denote by $|I|$ the normalized arc length of $I$ so that $|\partial D|=1$.
Let $S(I)$ be the Carleson box defined by
$$
S(I)=\{z: 1-|I|<|z|<1,\ z/|z|\in I\}.
$$
Let $0\le s<\infty$ and $0<p<\infty$.
For a nonnegative Borel measure $\mu$ on the unit disk $D$,
we define $T^{\infty}_{p,s}(\mu)$ as the space of
all $\mu$-measurable functions $f$ on $D$ satisfying
$$
\|f\|_{T^{\infty}_{p,s}(\mu)}^p
=\sup_{I\subset \partial D}|I|^{-s}\int_{S(I)}|f(z)|^p\,d\mu(z)<\infty.
$$
By a standard argument we can show that for $p\ge 1$,
$T^{\infty}_{p,s}(\mu)$ is a Banach space.

Next, we introduce the Carleson measures needed in this paper.
We say that a nonnegative Borel measure $\mu$ on $D$ is a
\textit{$(p,s)$-logarithmic Carleson measure}
if there is a constant $C>0$ such that
$$
\sup_{I\subset\partial D}
\frac1{|I|^s}\left(\log\frac 2{|I|}\right)^p\mu(S(I))
\le C.
$$
We will denote the class of all $(p,s)$-logarithmic Carleson measures on $D$
by $LCM_{p,s}$,
and denote by
$$
\|\mu\|_{LCM_{p,s}}=\sup_{I\subset\partial D}
\frac1{|I|^s}\left(\log\frac 2{|I|}\right)^p\mu(S(I)).
$$
When $p=0$, the $(0,s)$-logarithmic Carleson measures are called $s$-Carleson measures,
and we will denote by $CM_s=LCM_{0,s}$ and $\|\mu\|_{CM_s}=\|\mu\|_{LCM_{0,s}}$.
We say $\mu$ is a vanishing $(p,s)$-logarithmic Carleson measure if
$$
\lim_{|I|\to0}
\frac1{|I|^s}\left(\log\frac 2{|I|}\right)^p\mu(S(I))=0.
$$
A vanishing $(0,s)$-logarithmic Carleson measure
is also called a vanishing $s$-Carleson measure.

Logarithmic Carleson measures were first introduced by the second
author in \cite{zhao3}.
In this paper, we are going to characterize the measures $\mu$
such that the identity operator $I:\,F(p,q,s)\to T^{\infty}_{p,s}(\mu)$
is bounded or compact.
The results will be applied to give characterizations for the bounded and compact
Riemann-Stieltjes operators and pointwise multiplication operators on the $F(p,q,s)$ spaces.
Our results generalize some recent results by Xiao in \cite{xiao2}
and by Pau and Pel\'aez in \cite{pp}.

The paper is organized as follows. Section 2 is devoted to some preliminary results. In section 3 we characterize
boundedness and compactness of $I:F(p,p-2,s)\to T^{\infty}_{p,s}(\mu)$.
In section 4 we characterize
boundedness and compactness of $I:F(p,p\alpha-2,s)\to T^{\infty}_{p,s}(\mu)$
for $\alpha\neq1$.
In section 5 we use these results to characterize bounded and compact
Riemann-Stieltjes operators on $F(p,q,s)$, and in section 6 we characterize
bounded and compact pointwise multiplication operators on $F(p,q,s)$.

\bigskip

\section{Preliminary results}

In this Section we state and prove some preliminary results needed
for the rest of the paper. Some of them may have independent
interest.
We begin with the following lemma.

\begin{lemma}\label{1}
Let $0< p<\infty$, $-2<q<\infty$, $0\le s<\infty$ satisfy $q+s>-1$.
Let $\mu$ be a nonnegative Borel measure on $D$ such that the point evaluation
is a bounded functional on $T^{\infty}_{p,s}(\mu)$.
Then $I:\,F(p,q,s)\to T^{\infty}_{p,s}(\mu)$ is a compact operator if and only if
$\|f_n\|_{T^{\infty}_{p,s}}(\mu)\to0$ whenever
$\{f_n\}$ is a bounded sequence in $F(p,q,s)$ that converges to $0$
uniformly on every compact subset of $D$.
\end{lemma}

Using the fact that the point evaluations on
$F(p,q,s)$ and $T^{\infty}_{p,s}(\mu)$ are bounded functionals
(see Proposition 2.17 of \cite{zhao1} for the case of $F(p,q,s)$) ,
the proof of Lemma~\ref{1} is standard.
See, for example, the proof of Proposition 3.11 in \cite{cm}. We omit the details here.\\

We also need the following equivalent description of
$(p,s)$-logarithmic Carleson measures proved by O. Blasco (see Lemma
$4.1$ and Proposition $1.2$ in \cite{blasco}).

\begin{lemma}\label{4b} Let $s,t>0$ and $p\geq 0$. Let $\mu$ be a nonnegative Borel measure on $D$.
Then $\mu$ is a $(p,s)$-logarithmic Carleson measure if and only if
$$
\sup_{a\in D}\left (\log \frac{2}{1-|a|^2}\right )^p \int_D
\frac{(1-|a|^2)^t}{|1-\bar{a}z|^{s+t}}\,d\mu(z)<\infty.
$$
Further, we have
$$
\|\mu\|_{LCM_{p,s}} \approx \sup_{a\in D}\left (\log
\frac{2}{1-|a|^2}\right )^p \int_D
\frac{(1-|a|^2)^t}{|1-\bar{a}z|^{s+t}}\,d\mu(z).
$$
\end{lemma}

We also need the following equivalent definition for $T_{p,s}^{\infty}(\mu)$.

\begin{propo}\label{16} Let $\mu$ be a nonnegative Borel measure on $D$.
Let $p>0$, $s> 0$, and $t>0$. Then an analytic function $f$ on $D$ belongs to $T_{p,s}^{\infty}(\mu)$
if and only if
$$
\sup_{a\in D}\int_D\frac{(1-|a|^2)^t}{|1-\bar{a}z|^{s+t}}\,|f(z)|^p\,d\mu(z)<\infty,
$$
and
$$
\|f\|_{T_{p,s}^{\infty}(\mu)}^p\approx \sup_{a\in D}\int_D\frac{(1-|a|^2)^t}{|1-\bar{a}z|^{s+t}}\,|f(z)|^p\,d\mu(z).
$$
\end{propo}

\begin{proof}
Let $d\mu_f(z)=|f(z)|^p\,d\mu(z)$. Then
$$
\|f\|_{T_{p,s}^{\infty}(\mu)}^p
=\sup_{I\subset\partial D}\frac{1}{|I|^s}\int_{S(I)}|f(z)|^p\,d\mu(z)
=\sup_{I\subset \partial D}\frac{\mu_f(S(I))}{|I|^s}.
$$
Applying Lemma~\ref{4b} to $d\mu_f(z)$ we immediately get the
result.
\end{proof}

The following lemma is from \cite{of}.

\begin{lemma}\label{6}
For $s>-1$, $r,t>0$ with $r+t-s-2>0$, there is a constant $C>0$
such that
\begin{eqnarray*}
&~&\int_D\frac{(1-|z|^2)^s}{|1-\bar az|^r|1-\bar bz|^t}\,dA(z)\\
&~&\qquad\le
\begin{cases}
\displaystyle
\frac{C}{|1-\bar ab|^{r+t-s-2}},& \textrm{if \,$r,t<2+s$},\\
\displaystyle\frac{C}{(1-|a|^2)^{r-s-2}|1-\bar ab|^t},& \textrm{if \,$t<2+s<r$},\\
\displaystyle\frac{C}{(1-|a|^2)^{r-s-2}|1-\bar ab|^t}+\frac{C}{(1-|b|^2)^{t-s-2}|1-\bar ab|^r},
& \textrm{if \,$r,t>2+s$.}
\end{cases}
\end{eqnarray*}
\end{lemma}

\begin{coro}\label{7}
For $s>-1$, $r,t>0$ with $0<r+t-s-2<r$,
we have
$$
I(a,b)=\int_D\frac{(1-|z|^2)^s}{|1-\bar az|^r|1-\bar bz|^t}\,dA(z)
\le
\frac{C}{(1-|a|^2)^{r+t-s-2}}.
$$
\end{coro}

\begin{proof}
From the condition $0<r+t-s-2<r$ we know that $t<2+s$.
If $r<2+s$ or $r>2+s$ then the result follows
directly from Lemma~\ref{6}.\\
For the remaining case $r=2+s$, we cannot directly use Lemma~\ref{6}.
However, since $s+1>0$ and $t>0$, we can choose a number
$\delta$ such that $0<\delta<\min(t,s+1)$.
Thus
\begin{eqnarray*}
I(a,b)
&=&\int_D\frac{(1-|z|^2)^{s-\delta}}{|1-\bar az|^r|1-\bar bz|^{t-\delta}}
\cdot\frac{(1-|z|^2)^{\delta}}{|1-\bar bz|^{\delta}}\,dA(z)\\
&\le& C\int_D\frac{(1-|z|^2)^{s-\delta}}{|1-\bar az|^r|1-\bar bz|^{t-\delta}}\,dA(z).
\end{eqnarray*}
It is obvious that
$s-\delta>-1$, $r>0$, $t-\delta>0$, $r+(t-\delta)-(s-\delta)-2=r+t-s-2>0$.
Since $t<2+s$ and $r=2+s$ we also have
$$
t-\delta<2+(s-\delta)=r-\delta<r.
$$
Hence we can use the second inequality of Lemma~\ref{6} to get
$$
I(a,b)\le \frac{C}{(1-|a|^2)^{r-(s-\delta)-2}|1-\bar ab|^{t-\delta}}
\le\frac{C}{(1-|a|^2)^{r+t-s-2}}.
$$
The proof is complete.
\end{proof}

Using this corollary, we are able to prove the following result.

\begin{lemma}\label{8}
Let $0<p<\infty$, $0<\alpha<\infty$ and $0<s<\infty$ satisfy $s+p\alpha>1$. Let
$$
f_{b,\alpha}(z)=
\left\{ \begin{array}{ll}
\displaystyle \frac1{1-\alpha}(1-\bar bz)^{1-\alpha}, & \textrm{if $\alpha\neq1$},\\
\displaystyle \log\frac2{1-\bar bz}, & \textrm{if $\alpha=1$}.
\end{array} \right.
$$
Then
$$
\sup_{b\in D}\|f_{b,\alpha}\|_{F(p,p\alpha-2,s)}<\infty.
$$
\end{lemma}

\begin{proof}
A simple computation
using the well-known identity
$$
1-|\p_a(z)|^2=(1-|z|^2)|\p_a'(z)|
$$
shows
\begin{eqnarray*}
\|f_{b,\alpha}\|_{F(p,p\alpha-2,s)}^p
&=&\sup_{a\in D}\int_D\left|\frac{\bar b}{(1-\bar bz)^{\alpha}}\right|^{p}(1-|z|^2)^{p\alpha-2}(1-|\p_a(z)|^2)^s\,dA(z)\\
&=&\sup_{a\in D}|b|^p(1-|a|^2)^s\int_D\frac{(1-|z|^2)^{s+p\alpha-2}}{|1-\bar az|^{2s}|1-\bar bz|^{p\alpha}}\,dA(z).
\end{eqnarray*}
Since
$s+p\alpha-2>-1$, $p\alpha>0$, $2s>0$, and
$$
0<2s+p\alpha-(s+p\alpha-2)-2=s<2s,
$$
we can use Corollary~\ref{7} to get
$$
\int_D\frac{(1-|z|^2)^{s+p\alpha-2}}{|1-\bar az|^{2s}{|1-\bar bz|^{p\alpha}}}\,dA(z)
\le \frac{C}{(1-|a|^2)^s}.
$$
Therefore,
$$
\sup_{a\in D}\|f_{b,\alpha}\|_{F(p,p\alpha-2,s)}^p\le C<\infty.
$$
The proof is complete.
\end{proof}

\begin{lemma}\label{9}
Let $0<p<\infty$, and $\gamma>-1$ with $\gamma>p-2$. For any $f\in
D^p_{\gamma}$, and any $a\in D$,
\begin{displaymath}
I(a):=\int_{D}\frac{|f(z)-f(0)|^p}{|1-\bar{a}z|^{p}}\,(1-|z|^2)^{\gamma}dA(z)\leq
C\, \int_{D}|f'(z)|^p (1-|z|^2)^{\gamma} dA(z).
\end{displaymath}
\end{lemma}
\begin{proof}
If $p>1$, the result follows from Lemma $2.1$ of \cite{bp}. If
$0<p\leq 1$ we use the atomic decomposition for $D^p_{\gamma}$
(see \cite[Theorem 32]{ZZ}): there is a sequence $\{a_ k\}$ in $D$
and a sequence of numbers $\{\lambda_ k\}\in \ell^p$ such that
\begin{displaymath}
g(z)=f(z)-f(0)=\sum_ k \lambda_ k \,\frac{(1-|a_
k|^2)^{b-(2+\gamma-p)/p}}{(1-\bar{a}_ k z)^b},\qquad
b>\frac{2+\gamma-p}{p}
\end{displaymath}
and
\begin{displaymath}
\sum_ k |\lambda_ k|^p \leq C \|g\|^p_{D^p_{\gamma}}=C \int_{D}
|f'(z)|^p (1-|z|^2)^{\gamma}dA(z).
\end{displaymath}
Note that $D^p_{\gamma}=A^p_{\gamma-p}$ in the notation of
\cite{ZZ}. Now, since $0<p\leq 1$, we obtain
\begin{displaymath}
\begin{split}
I(a)&\leq \sum_ k |\lambda_ k|^p(1-|a_
k|^2)^{pb-2-\gamma+p}\int_{D}\frac{(1-|z|^2)^{\gamma}}{|1-\bar{a}_
k z|^{pb}|1-\bar{a}z|^{p}}\,dA(z)
\\
&\leq C \sum_ k |\lambda_ k|^p.
\end{split}
\end{displaymath}
The last inequality holds by Corollary \ref{7}. The proof is
complete.
\end{proof}

\begin{propo}\label{10}
Let $0<p<\infty$, $s>0$, and $\alpha\ge 1$ with $\alpha p+s>1$.
Let $t=s+2p(\alpha-1)$. For any $f\in F(p,p\alpha-2,s)$ and any
$a\in D$,
$$
J(a)=(1-|a|^2)^t\int_{D}\frac{|f(z)-f(a)|^p}{|1-\bar{a}z|^{p+s+t}}
\,(1-|z|^2)^{s+p\alpha-2}dA(z)\leq C \,\|f\|^p_{F(p,p\alpha-2,s)}.
$$
\end{propo}

\begin{proof}
Changing the variable $z=\p_a(w)$, we get by Lemma~\ref{9},
\begin{eqnarray*}
J(a)&=&(1-|a|^2)^{p(\alpha-1) }\int_{D}\frac{|f\circ
\varphi_a(w)-f\circ \varphi_ a(0)|^p}{|1-\bar{a}w|^{p} }
(1-|w|^2)^{s+p\alpha-2}dA(w)
\\
&\le &C (1-|a|^2)^{p(\alpha-1)}\int_{D}\big |
(f\circ \varphi_a)'(w)\big |^p (1-|w|^2)^{s+p\alpha-2}dA(w)
\\
&=&C (1-|a|^2)^{p(\alpha-1)}\int_{D} |f'(z)|^p
(1-|z|^2)^{p-2}(1-|\varphi_ a(z)|^2)^{s+p(\alpha-1)}dA(z)
\\
&=&C(1-|a|^2)^t \int_{D} |f'(z)|^p\frac{(1-|z|^2)^{s+p\alpha-2}}
{|1-\bar{a}z|^{s+t}}dA(z)\leq C \|f\|^p_{F(p,p\alpha-2,s)}.
\end{eqnarray*}

The proof is complete.
\end{proof}

\begin{propo}\label{pn}
Let $0\le p<\infty$ and $s\ge 1$. Let $\mu$ be a vanishing
$(p,s)$-logarithmic Carleson measure on $D$. Then
$$
\lim_{r\to1}\sup_{I\subset\partial D}\frac1{|I|^s}\left(\log\frac2{|I|}\right)^p\mu_{r}(S(I))=0.
$$
\end{propo}
\begin{proof}
Let $0\le p<\infty$ and $s\ge 1$. Since $\mu$ is a vanishing
$(p,s)$-logarithmic Carleson measure, for any $\varepsilon>0$,
there is an $r$, $0<r<1$, such that
$$
\frac1{|I|^s}\left(\log\frac2{|I|}\right)^p\mu(S(I))<\varepsilon
$$
for any arc $I$ with $|I|<1-r$.
Now fix the above $r$. Consider any arc $I$ on the unit circle $\partial D$.
Let $\alpha=1-r$, and $n=[|I|/\alpha]$. Then $n\alpha\le |I|<(n+1)\alpha$.
Clearly, we can cover $I$ by $n+1$ arcs $I_1$, $I_2$, ..., $I_{n+1}$ with $|I_k|=1-r=\alpha$
for $k=1,2,...,n+1$.
Let $\mu_r=\mu|_{D\setminus D_r}$.
Then
\begin{eqnarray*}
\frac1{|I|^s}\left(\log\frac2{|I|}\right)^p\mu_{r}(S(I))
&\le&
\sum_{k=1}^{n+1}\frac1{|I|^s}\left(\log\frac2{|I|}\right)^p\mu(S(I_k))
\\
&\le&
 \varepsilon\sum_{k=1}^{n+1}\frac1{|I|^s}
\left(\log\frac2{|I|}\right)^p|I_k|^s\left(\log\frac2{|I_k|}\right)^{-p}
\\
&\le&\varepsilon\sum_{k=1}^{n+1}\frac1{(n\alpha)^s}
\left(\log\frac2{n\alpha}\right)^p\alpha^s\left(\log\frac2{\alpha}\right)^{-p}
\\
&\le&2n^{1-s}\varepsilon\left[\left(\log\frac2{n\alpha}\right)
\big/\left(\log\frac2{\alpha}\right)\right]^p
\le C\varepsilon
\end{eqnarray*}
uniformly on $I$, when $s\ge 1$. Hence, as $s\ge 1$ we have
$$
\lim_{r\to1}\sup_{I\subset\partial D}\frac1{|I|^s}\left(\log\frac2{|I|}\right)^p\mu_{r}(S(I))=0.
$$
\end{proof}

\begin{section}
{Embedding from $F(p,p-2,s)$ to $T^{\infty}_{p,s}(\mu)$}
\end{section}

In this section we describe the boundedness and compactness of the embedding from the M\"{o}bius
invariant space $F(p,p-2,s)$ to the tent-type space $T^{\infty}_{p,s}(\mu)$.
The main result in this section is the following.

\begin{theorem}\label{2}
Let $s>0$ and $0<p<\infty$ satisfy $s+p>1$, and let $\mu$
be a nonnegative Borel measure on $D$.
Then
\begin{itemize}
\item[(i)]
$I:F(p,p-2,s)\to T^{\infty}_{p,s}(\mu)$ is bounded if and only if
$\mu$ is a $(p,s)$-logarithmic Carleson measure.
\item[(ii)] The following two conditions are equivalent:
\begin{itemize}
\item[(a)]
For any bounded sequence $\{f_n\}$ in $F(p,q,s)$ satisfying $f_n(z)\to0$ uniformly
on every compact subset of $D$,
$$
\lim_{n\to\infty}\|f_n\|_{T_{p,s}^{\infty}(\mu)}=0.
$$
\item[(b)] $\mu$ is a vanishing $(p,s)$-logarithmic Carleson measure.
\end{itemize}
\end{itemize}
\end{theorem}

\noindent
\textit{Remark.} When $p=2$, the result is obtained by Xiao in \cite{xiao2} using techniques from \cite{pp}.
\medskip

In order to prove Theorem~\ref{2},
we need some results on Carleson measures for weighted Dirichlet spaces $D^p_{\alpha}$.
A nonnegative Borel measure $\mu$ on $D$
is said to be a Carleson measure for $D^p_{\alpha}$
if there is a constant $C>0$ such that for every $f\in D^p_{\alpha}$,
$$
\int_D|f(z)|^p\,d\mu(z)\le C\|f\|_{D^p_{\alpha}}^p,
$$
and is called a compact Carleson measure for $D^p_{\alpha}$ if $\lim_ n\|f_ n\|_{L^p(\mu)}=0$
whenever $\{f_ n\}$ is a bounded sequence in $D^p_{\alpha}$ converging
to zero uniformly on compact subsets of $D$.

Carleson measures for the
weighted Bergman spaces $A^p_{\alpha}=D^p_{\alpha+p}$, with $\alpha>-1$ and $p>0$ are described
as those positive Borel measures on $D$ such that
$$\sup_{I\subset \partial D} \frac{\mu(S(I))}{|I|^{2+\alpha}}<\infty.$$ This result was proved by several authors,
including Oleinik and Pavlov \cite{op} (for $p>1$), Stegenga
\cite{stegenga2} (for $\alpha=0$), and Hastings \cite{hastings}. One
can also find a proof in \cite{luecking}. The next result will be essential
in order to prove Theorem \ref{2}.
\begin{lemma}\label{5}
Let $p>0$, $s\ge 0$ satisfy $s+p>1$.
Let $\mu$ be a nonnegative Borel measure on $D$.
\begin{enumerate}
\item[(a)]
If $\mu\in LCM_{p,s}$ then
$\mu$ is a Carleson measure for $D^p_{s+p-2}$. Moreover
$$\big\|I:D^p_{s+p-2}\rightarrow L^p(\mu)\big\|^p\le C \,\|\mu\|_{LCM_{p,s}}.$$
\item [(b)] If $\mu$ is a vanishing $(p,s)$-logarithmic Carleson measure, then $\mu$ is a compact Carleson measure for $D^p_{s+p-2}$.
\end{enumerate}
\end{lemma}

\begin{proof}
The case $p=2$ with $0<s<1$ of part (a) was proved in \cite{pp} using the description of Carleson measures for $D^p_{s+p-2}$ given in \cite{ars}. Here we will prove the result directly from the definition of Carleson measures for $D^p_{s+p-2}$.

We first consider the case $p>1$. Let $p'$ denote the conjugate exponent of $p$. Since $p'>1$ is easy to see that
$$
\int_{D} \frac{dA(w)}{|1-\bar{w}z|^2\big (\log \frac{2}{1-|w|^2}\big
)^{p'}}\leq C
$$
for some constant $C$ independent of $w\in D$. Let $\alpha=s+p-2$,
and $f\in D^p_{\alpha} $ with $f(0)=0$.  Using the reproducing
formula for the Bergman space $A^p_{\alpha}$ (see p.80 of
\cite{zhu1}), and then integrating along the segment $[0,z]$ it
follows that
\begin{displaymath}
f(z)=C\int_D f'(w)K_{\alpha}(w,z)(1-|w|)^{\alpha}\,dA(w),
\end{displaymath}
where
\begin{displaymath}
K_{\alpha}(w,z)=\frac{1-(1-\bar wz)^{1+\alpha}}{\bar w(1-\bar
wz)^{1+\alpha}}.
\end{displaymath}
It is easy to check that
$$
\sup_{z,w\in D}\left|\frac{1-(1-\bar wz)^{1+\alpha}}{\bar
w}\right|\le C
$$
for some constant $C>0$. Then
\begin{displaymath}
\begin{split}
|f(z)|^p
&\le C \left (\int_{D}
\frac{|f'(w)|\,(1-|w|^2)^{\alpha}}{|1-\bar{w}z|^{1+\alpha}}\,dA(w)\right
)^p
\\
&\leq C \int_{D}
\frac{|f'(w)|^p\,(1-|w|^2)^{p\alpha}}{|1-\bar{w}z|^{p\alpha+2-p}}\big
(\log \frac{2}{1-|w|^2}\big )^p\,dA(w)\left (\int_{D}\frac{(\log
\frac{2}{1-|w|^2})^{-p'}dA(w)}{|1-\bar{w}z|^2}\right )^{p-1}
\\
&
\le C \int_{D}
\frac{|f'(w)|^p\,(1-|w|^2)^{p\alpha}}{|1-\bar{w}z|^{p\alpha+2-p}}\big (\log
\frac{2}{1-|w|^2}\big )^p\,dA(w).
\end{split}
\end{displaymath}
This inequality together with Fubini's theorem and Lemma \ref{4b}
yields
\begin{displaymath}
\begin{split}
\int_{D}|f(z)|^p &d\mu(z)
\\
&\le C\int_{D}|f'(w)|^p\,(1-|w|^2)^{p\alpha}\left (\log
\frac{2}{1-|w|^2}\right
)^p\int_{D}\frac{d\mu(z)}{|1-\bar{w}z|^{p\alpha+2-p}}\,dA(w)
\\
&\le C \|\mu\|_{LCM_{p,s}}
\int_{D}|f'(w)|^p\,(1-|w|^2)^{p\alpha}\,(1-|w|^2)^{-(p-1)\alpha}\,dA(w)
\\
&\le C \|\mu\|_{LCM_{p,s}}
\|f\|^p_{D^p_{\alpha}}.
\end{split}
\end{displaymath}
Hence $\mu$ is a Carleson measure for $D^p_{s+p-2}$.

Now we consider the case $0<p\le 1$. For $r>0$, fix an $r$-lattice
$\{a_ n\}$ in the Bergman metric. This means that the hyperbolic
disks $D(a_ n,r)=\{z:\beta(z,a_ n)<r\}$ cover the unit disk $D$
and $ \beta(a_ i, a_ j)\ge r/2$ for all $i$ and $j$ with $i \ne
j$. Here $\beta(z,w)$ denotes the Bergman or hyperbolic metric. If
$\{a_ k\}$ is an $r$-lattice in $D$, then it also has the
following property: for any $R > 0$ there exists a positive
integer $N$ (depending on $r$ and $R$) such that every point in
$D$ belongs to at most $N$ sets in $\{D(a_ k,R)\}$. There are
elementary constructions of $r$-lattices in $D$. See \cite[Chapter
4]{zhu1} for example. Note that by subharmonicity we have
\begin{displaymath}
\sup\{|f'(w)|:w\in D(a_ j,r)\}\le C \left (\frac{1}{(1-|a_
j|^2)^2}\int_{D(a_ j,2r)}|f'(\zeta)|^p\,dA(\zeta)  \right )^{1/p}
\end{displaymath}
for all $j=1,2,\dots$ Let $\beta$ be a sufficiently large number
so that $\beta\ge s+p-2$ and $\beta p>s-p$. It follows that
\begin{displaymath}
\begin{split}
|f(z)|^p &\le C \left (\int_{D}
\frac{|f'(w)|\,(1-|w|^2)^{\beta}}{|1-\bar{w}z|^{1+\beta}}\,dA(w)\right
)^p
\\
& \le C \left (\sum_ j \int_{D(a_ j,r)}
\frac{|f'(w)|\,(1-|w|^2)^{\beta}}{|1-\bar{w}z|^{1+\beta}}\,dA(w)\right
)^p
\\
&\le C \left (\sum_ j \frac{(1-|a_ j|^2)^{2+\beta}}{|1-\bar{a_
j}z|^{1+\beta}}\left (\frac{1}{(1-|a_ j|^2)^2}\int_{D(a_
j,2r)}|f'(\zeta)|^p\,dA(\zeta)  \right )^{1/p}\right )^p
\\
&\le C \sum_ j \frac{(1-|a_ j|^2)^{2p-2+\beta p}}{|1-\bar{a}_ j
z|^{p+\beta p}}\int_{D(a_ j,2r)}|f'(\zeta)|^p\,dA(\zeta)
\end{split}
\end{displaymath}
Now, if $\mu\in LCM_{p,s}$, then clearly $\mu\in CM_ s$ with
$\|\mu\|_{CM_ s}\le C \|\mu\|_{LCM_{p,s}}$. Therefore, since
$1-|a_ j|^2$ is comparable to $1-|\zeta|^2$ if $\zeta\in D(a_
j,2r)$, an application of Lemma \ref{4b} gives
\begin{displaymath}
\begin{split}
\int_{D} |f(z)|^p\,d\mu(z)&\le C \sum_ j \left ((1-|a_
j|^2)^{2p-2+\beta p}\int_{D}\frac{d\mu(z)}{|1-\bar{a}_ j
z|^{p+\beta p}}\right)\int_{D(a_
j,2r)}\!\!\!|f'(\zeta)|^p\,dA(\zeta)
\\
& \le C \,\|\mu\|_{CM_ s}\,\sum_ j (1-|a_ j|^2)^{s+p-2}\int_{D(a_
j,2r)}|f'(\zeta)|^p\,dA(\zeta)
\\
& \le C \,\|\mu\|_{CM_ s}\,\sum_ j \int_{D(a_
j,2r)}|f'(\zeta)|^p\,(1-|\zeta|^2)^{s+p-2}\,dA(\zeta)
\\
& \le C\,N \,\|\mu\|_{CM_ s}\,\|f\|^p_{D^p_{s+p-2}},
\end{split}
\end{displaymath}
where $N$ is a positive integer such that each point of $D$
belongs to at
most $N$ of the sets $D(a_ j,2r)$. This finishes the proof of (a).\\

To prove (b), we must show that if $\{f_ n\}$ is a bounded
sequence in $D^p_{s+p-2}$ converging to zero uniformly on compact
subsets of $D$, then $\|f_ n\|_{L^p(\mu)}\rightarrow 0$. Suppose first that $p>1$. Let
$\alpha=s+p-2$. Since $\mu$ is a vanishing $(p,s)$-logarithmic
Carleson measure on $D$,  for any $\varepsilon>0$, there is an
$r$, $0<r<1$, such that
\begin{displaymath}
\sup_{|w|>r}\left (\log \frac{2}{1-|w|^2}\right
)^p\,(1-|w|^2)^{(p-1)\alpha}\int_{D}\frac{d\mu(z)}{|1-\bar{w}z|^{p\alpha+2-p}}<\varepsilon^p.
\end{displaymath}
Since $\{f_ n\}$ converges to zero uniformly on compact sets, the
same is true for the sequence of its derivatives $\{f_ n'\}$, and hence there is a
positive integer $n_ 0$ such that for all $n\ge n_ 0$
$$\sup_{|w|\le r}|f'_ n(w)|<\varepsilon.$$
Now, we have
\begin{displaymath}
\begin{split}
\|f_ n\|^p_{L^p(\mu)} &\le C \int_{D}\left (\int_{D} \frac{|f_
n'(w)|\,(1-|w|^2)^{\alpha}}{|1-\bar{w}z|^{1+\alpha}}\,dA(w)\right
)^p d\mu(z)
\\
&
= C\int_{D}\left (\int_{|w|\le r}+\int_{|w|>r}\right
)^p d\mu(z)
\end{split}
\end{displaymath}
For the first term, we have
$$\int_{D}\left (\int_{|w|\le r} \dots\right )^pd\mu(z) <\varepsilon^p
\int_{D}\left (\int_{D}
\frac{(1-|w|^2)^{\alpha}}{|1-\bar{w}z|^{1+\alpha}}\,dA(w)\right
)^pd\mu(z)\le C \varepsilon^p.$$ For the second term, proceeding
as in the proof of part (a), we have
\begin{displaymath}
\begin{split}
\int_{D}&\left (\int_{|w|> r}\dots \right )^p \!\! d\mu(z)
\\
&\le C \! \int_{|w|>r}\!|f_ n'(w)|^p\,(1-|w|^2)^{p\alpha}\!\left
(\log \frac{2}{1-|w|^2}\right )^p
\!\!\int_{D}\frac{d\mu(z)}{|1-\bar{w}z|^{p\alpha+2-p}}\,dA(w)
\\
& <C\varepsilon^p \|f_ n\|_{D^p_{s+p-2}}^p\le C \varepsilon^p.
\end{split}
\end{displaymath}
This finishes the proof of (b) for $p>1$. For $0<p\le 1$, we
observe first that the condition on $\mu$ implies that $\mu$ is a
vanishing $s$-Carleson measure. Therefore, given $\varepsilon>0$,
there is an $r_ 0$ with $0<r_ 0<1$ such that
\begin{equation}\label{Ep1}
\sup_{|a|>r_ 0} \int_{D} \frac{(1-|a|^2)^{p+\beta
p-s}}{|1-\bar{a}z|^{p+\beta p}}\,d\mu(z)<\varepsilon,
\end{equation}
 where $\beta$ is a sufficiently large number
so that $\beta\ge s+p-2$ and $\beta p>s-p$. \\Let $\{f_ n\}$ be a
bounded sequence in $D^p_{s+p-2}$ converging to zero uniformly on
compact subsets of $D$, and for $r>0$ fix an $r$-lattice $\{a_
k\}$ in the Bergman metric. Since $|a_ k|\rightarrow 1$, there are
only $k_ 0$ points $a_ k$ with $|a_ k|\le r_ 0$. Now, the same
argument used in part (a) gives
\begin{displaymath}
\begin{split}
\|f_ n\|_{L^p(\mu)}^p& \le C \,\|\mu\|_{CM_ s}\,\sum_ {k=1}^{k_ 0}
(1-|a_ k|^2)^{s+p-2}\int_{D(a_ k,2r)}|f'_ n(\zeta)|^p\,dA(\zeta)
\\
&  +C \!\!\sum_ {|a_ k|>r_ 0} \left ((1-|a_ k|^2)^{2p-2+\beta
p}\int_{D}\frac{d\mu(z)}{|1-\bar{a}_ k z|^{p+\beta
p}}\right)\int_{D(a_ k,2r)}\!\!\!|f'_ n(\zeta)|^p\,dA(\zeta)
\\
&=(I)+(II).
\end{split}
\end{displaymath}
Since $\{f_ n\}$ converges to zero uniformly on compact sets, the
same is true for the sequence of its derivatives $\{f_ n'\}$, and
hence there is a positive integer $n_ 0$ such that for all $n\ge
n_ 0$
$$\sup \big\{|f'_ n(\zeta)|^p:\zeta\in D(a_ k,2r)\big \}<\frac{\varepsilon}{k_ 0},\qquad 1\le k\le k_ 0.$$
This gives
\begin{displaymath}
(I)\le C \,\|\mu\|_{CM_ s}\,\varepsilon.
\end{displaymath}
Also, using \eqref{Ep1}, the same proof given in part (a) yields
\begin{displaymath}
(II)\le C \,N\,\varepsilon\,\|f_ n\|^p_{D^p_{s+p-2}}\le
CN\varepsilon,
\end{displaymath}
where $N$ is a positive integer such that each point of the unit
disk belongs to at most $N$ of the sets $D(a_ k,2r)$. This
finishes the proof of the lemma.
\end{proof}

Now we are ready to prove our main result in this section.

\subsection{Proof of Theorem~\ref{2}}
We first prove (i). Suppose $\mu$ is a $(p,s)$-logarithmic
Carleson measure. Let $s>0$ and $p>0$ satisfy $s+p>1$. By
Lemma~\ref{5}, $\mu$ is a Carleson measure for $D^p_{s+p-2}$.

Given any subarc $I$ of $\partial D$, let $w=(1-|I|)\zeta$ and
$\zeta$ be the center on $I$. An easy computation shows that, for
any $z\in S(I)$,
$$
1-|w|^2\approx |1-\bar wz|\approx |I|.
$$
Take any function $f\in F(p,p-2,s)$.
By Corollary 2.8 in \cite{zhao1},
$F(p,p-2,s)\subset B$. Thus, from a well-known estimate of Bloch functions
(see, for example \cite{zhu1}),
\begin{displaymath}
|f(w)|
\le |\|f\||_B\log\frac2{1-|w|^2}
\le C|\|f\||_{F(p,p-2,s)}\log\frac2{1-|w|^2}
\le C|\|f\||_{F(p,p-2,s)}\log\frac2{|I|}.
\end{displaymath}
This, together with the fact that $\mu\in LCM_{p,s}$, gives
\begin{eqnarray*}
&~&\frac{1}{|I|^s}\int_{S(I)}|f(z)|^p\,d\mu(z)
\\
&~&\qquad\le \frac{C}{|I|^s}\left(\int_{S(I)}|f(z)-f(w)|^p\,d\mu(z)+|f(w)|^p\,\mu(S(I))\right)
\\
&~&\qquad\le C(1-|w|^2)^s\int_D\left|\frac{f(z)-f(w)}{(1-\bar wz)^{2s/p}}\right|^p\,d\mu(z)
+C\|f\|_{F(p,p-2,s)}^p\|\mu\|_{LCM_{p,s}}.
\end{eqnarray*}
On the other hand, if we let
$$
f_ w(z)=(1-|w|^2)^{s/p}\,\frac{f(z)-f(w)}{(1-\bar wz)^{2s/p}},
$$
then, applying Lemma~\ref{5},  we get

\begin{displaymath}
\begin{split}
 (1-|w|^2)^s\int_{D}\left|\frac{f(z)-f(w)}{(1-\bar wz)^{2s/p}}\right|^p\,d\mu(z)&=\int_{\D} |f_ w(z)|^p\,d\mu(z)
\\
&\le C \|\mu\|_{LCM_{p,s}}\,\|f_ w\|^p_{D^p_{s+p-2}}
\\
&\le C \|\mu\|_{LCM_{p,s}}\,\|f\|^p_{F(p,p-2,s)},
\end{split}
\end{displaymath}
since $\|f_ w\|^p_{D^p_{s+p-2}}\le C \|f\|^p_{F(p,p-2,s)}$, with $C$ being a positive constant independent of $w$. Indeed,  notice first that
$$(1-|w|^2)^s\,|f(w)-f(0)|^p\le C\,\|f\|_{F(p,p-2,s)}\,(1-|w|^2)^s\,\big (\log \frac{2}{1-|w|^2}\big )^p\le C\,\|f\|_{F(p,p-2,s)},$$
and this, together with Proposition~\ref{10}, yields
\begin{displaymath}
\begin{split}
\|f_ w\|^p_{D^p_{s+p-2}}&=|f_ w(0)|^p
+\int_ D|f'_ w(z)|^p(1-|z|^2)^{s+p-2}\,dA(z)
\\
&\le (1-|w|^2)^s\,|f(w)-f(0)|^p
\\
&+C\,(1-|w|^2)^s\int_D\frac{|f'(z)|^p}{|1-\bar wz|^{2s}}(1-|z|^2)^{s+p-2}\,dA(z)
\\
&+C\,(1-|w|^2)^s\int_D
\frac{|f(z)-f(w)|^p}{|1-\bar wz|^{2s+p}}(1-|z|^2)^{s+p-2}\,dA(z)
\\
&\le C \,\|f\|^p_{F(p,p-2,s)}.
\end{split}
\end{displaymath}

This shows that the identity operator $I:F(p,p-2,s)\to T^{\infty}_{p,s}(\mu)$
is a bounded operator.\\

Conversely, suppose that the identity operator $I:F(p,p-2,s)\to T^{\infty}_{p,s}(\mu)$
is bounded. Given any arc $I\subset \partial D$, let $a=(1-|I|)\zeta$, where $\zeta$ is
the center of $I$, and consider the function
$$
f_a(z)=\log\frac{2}{1-\bar az}.
$$
Then by Lemma~\ref{8}, there is a constant $C>0$ such that
$
|\|f_a\||_{F(p,p-2,s)}\le C.
$
Since $I:F(p,p-2,s)\to T^{\infty}_{p,s}(\mu)$ is bounded,
we get
$$
\frac1{|I|^s}\int_{S(I)}|f_a(z)|^p\,d\mu(z)\le C|\|f_a\||_{F(p,p-2,s)}^p\le C.
$$
It is easy to see that
$$
f_a(z)\approx \log\frac2{|I|}
$$
for any $z\in S(I)$. Combining the above two inequalities we get
$$
\frac1{|I|^s}\left(\log\frac2{|I|}\right)^p\mu(S(I))\le C.
$$
Thus $\mu$ is a $(p,s)$-logarithmic Carleson measure, completing the proof of (i).\\

To prove (ii), first let $\mu$ be a vanishing $(p,s)$-logarithmic
Carleson measure and $\{f_n\}$ be a bounded sequence in
$F(p,p-2,s)$ with $f_n(z)\to 0$ uniformly on every compact subset
of $D$. We need to prove that
$\|f_n\|_{T^{\infty}_{p,s}(\mu)}\to0$.

The case $s\ge 1$ is easier. Let $D_r=\{z\in D:\,|z|\le r\}$, and
$\mu_r=\mu|_{D\setminus D_r}$ be the restriction of $\mu$ on
$D\setminus D_r$. By Proposition \ref{pn}, if $r\to1$ then
$$
\sup_{I\subset\partial
D}\frac1{|I|^s}\left(\log\frac2{|I|}\right)^p\mu_r(S(I)) \to0.
$$
Now, by part (i), the fact that the limit
$\|f_n\|_{T^{\infty}_{p,s}}(\mu)\to0$  follows from
$$
\int_{S(I)}|f_n(z)|^p\,d\mu(z) \le
C\int_{S(I)}|f_n(z)|^p\,d\mu_{D_r}(z)
+C\|f_n\|_{F(p,p-2,s)}^p\left(\log\frac2{|I|}\right)^p\mu_r(S(I)).
$$
For $0<s<1$ we don't have Proposition \ref{pn} at our disposal, so that
the proof must follow a different route that also works for $s\ge
1$. For any arc $I$,  let $w=w_ I=(1-|I|)\zeta$, with $\zeta$
being the center of $I$. Then we have
\begin{displaymath}
\frac{1}{|I|^s}\int_{S(I)} |f_ n(z)|^p\,d\mu(z) \le
\frac{C}{|I|^s}\left (\int_{S(I)}|f_ n(z)-f_ n(w)|^p\,d\mu(z)+|f_
n(w)|^p\mu(S(I))\right ).
\end{displaymath}
Since $\mu$ is a vanishing $(p,s)$-logarithmic Carleson measure on
$D$,  for any $\varepsilon>0$, there is an $r$, $0<r<1$, such that
$$
\sup_{|I|<1-r} \Big (\log \frac{2}{|I|}\Big )^p\, \frac{\mu \big(S(I)\big)}{|I|^s}<\varepsilon.
$$
This, together with the facts that $\{f_ n\}$ is a bounded
sequence in $F(p,p-2,s)$ and the inequality
$|f_ n(w)|\le C |\|f_n|\|_{F(p,p-2,s)}\log \frac{2}{|I|}$ gives
$$
\sup_{|I|<1-r} |f_ n(w)|^p \,\frac{\mu \big(S(I)\big)}{|I|^s}<C\varepsilon.
$$
If $|I|\ge 1-r$ then $|w|\le r$, and since $\{f_ n\}$ converges to
zero uniformly on compact subsets of $D$, there is a positive
integer $n_ 0$ such that
$$
\sup_{|I|\ge 1-r}|f_ n(w)|^p\,\frac{\mu \big(S(I)\big)}{|I|^s} <\varepsilon
\qquad \textrm{for } \quad n\ge n_ 0.
$$
For the other term, we have
\begin{displaymath}
\begin{split}
J_ n :=\frac{1}{|I|^s}&\int_{S(I)} |f_ n(z)-f_ n(w)|^p\,d\mu(z)
\\
& \le C (1-|w|^2)^s \int_{D} \left |\frac{f_ n(z)-f_n(w)}{(1-\bar{w}z)^{2s/p}}\right |^p\,d\mu(z)
\\
&=C\,\int_{D} |f_{n,w}(z)|^p\,d\mu(z),
\end{split}
\end{displaymath}
where
$$
f_{n,w}(z)=(1-|w|^2)^{\frac{s}{p}}\frac{(f_ n(z)-f_n(w))}{(1-\bar{w}z)^{\frac{2s}{p}}}.
$$
Fix a sufficiently large
number $\beta$ so that $\beta\ge s+p-2$ and $p(1+\beta)>2s$. The
proof of part (b) in Lemma \ref{5} gives
\begin{displaymath}
\begin{split}
 \|f_{n,w}\|_{L^p(\mu)}^p&\le C
|f_{n,w}(0)|^p\,\mu(D)+C\int_{D}\left (\int_{|\zeta|\le r}
\frac{|f_{
n,w}'(\zeta)|\,(1-|\zeta|^2)^{\beta}}{|1-\bar{\zeta}z|^{1+\beta}}\,dA(\zeta)\right
)^p d\mu(z)
\\
&+ C \,\varepsilon\,\|f_{n,w}\|^p_{D^p_{s+p-2}}.
\end{split}
\end{displaymath}
Since $\{f_ n\}$ is a bounded sequence in $F(p,p-2,s)$, it follows
from the proof of the boundedness part that
$$
\|f_{n,w}\|_{D^p_{s+p-2}}\le C \|f_ n\|_{F(p,p-2,s)}\le C,
$$
with $C$ independent of $n$ and $w$.

It is clear that
$$
f_{n,w}(0)=(1-|w|^2)^{\frac{s}{p}}(f_ n(0)-f_n(w)).
$$
It is well-known $F(p,p-2,s)\subset B$, the Bloch space.
Hence $\{f_n\}$ is a bounded sequence in $B$ and so there is a constant $K>0$
such that $\|f_n\|_{B}\le K$. Thus
$$
|f_{n,w}(0)|^p=(1-|w|^2)^{s}|f_ n(0)-f_n(w)|^p\le K^p(1-|w|^2)^{s}\left(\log\frac{2}{1-|w|}\right)^p,
$$
and so for any $\varepsilon>0$ and any $n\in\mathbb N$ there is an $r\in (0,1)$ such that
$$
|f_{n,w}(0)|^p<\varepsilon
$$
whenever $r<|w|<1$.
On the other hand, since $f_n(w)\to0$ uniformly on compact subsets of $D$, we know that
there is an $N>0$ such that if $n\ge N$ then
$$
|f_{n,w}(0)|^p=(1-|w|^2)^{s}|f_ n(0)-f_n(w)|^p<\varepsilon
$$
for all $|w|\le r$. Combining the above arguments we know that
$\sup_{w\in D}|f_{n,w}(0)|^p<\varepsilon$ if $n$ is
sufficiently large.

It remains only to deal with the
term
$$
A(n,w):=\int_{D}\left (\int_{|\zeta|\le r} \frac{|f_{
n,w}'(\zeta)|\,(1-|\zeta|^2)^{\beta}}{|1-\bar{\zeta}z|^{1+\beta}}\,dA(\zeta)\right)^p d\mu(z).
$$
We have that $$A(n,w)\le \max(1,2^{p-1})\Big( A_
1(n,w)+A_ 2(n,w)\Big )$$ with
$$A_ 1(n,w):=(1-|w|^2)^s\int_{D}\left (\int_{|\zeta|\le r} \frac{|f_{
n}'(\zeta)|\,(1-|\zeta|^2)^{\beta}}{|1-\bar{\zeta}z|^{1+\beta}\,|1-\bar{w}\zeta|^{\frac{2s}{p}}}\,dA(\zeta)\right
)^p d\mu(z),$$ and
$$A_ 2(n,w):=(1-|w|^2)^s\int_{D}\left (\int_{|\zeta|\le r} \frac{|f_{
n}(\zeta)-f_
n(w)|\,(1-|\zeta|^2)^{\beta}}{|1-\bar{\zeta}z|^{1+\beta}\,|1-\bar{w}\zeta|^{\frac{2s+p}{p}}}\,dA(\zeta)\right
)^p d\mu(z).$$ Since $\{f'_ n\}$ converges to zero uniformly on
compact subsets of $D$, it is clear that $$\sup_{w\in D}A_
1(n,w)<\varepsilon$$ for $n$ big enough. Similarly, if $0<r_ 0<1$
is fixed, it follows from the fact that $\{f_ n\}$ converges
uniformly to zero on compact subsets that
$$\lim_{n\rightarrow \infty}\sup_{|w|\le r_ 0} A_ 2(n,w)=0.$$
Since $\mu$ is a vanishing $(p,s)$-logarithmic Carleson measure,
we can choose $r_ 0$ so that
$$ \sup_{|w|>r_ 0} \left (\log \frac{2}{1-|w|^2}\right )^p\int_{D}
\frac{(1-|w|^2)^s}{|1-\bar{w}z|^{2s}}\,d\mu(z)<\varepsilon.$$ This
together with Lemma \ref{6} gives
\begin{displaymath}
\begin{split}
\sup_{|w|>r_ 0} (1-|w|^2)^s&\,|f_ n(w)|^p\int_{D}\left (\int_{D}
\frac{(1-|\zeta|^2)^{\beta}}{|1-\bar{\zeta}z|^{1+\beta}\,|1-\bar{w}\zeta|^{\frac{2s+p}{p}}}\,dA(\zeta)\right
)^p d\mu(z)
\\
&\le C \sup_{|w|>r_ 0} (1-|w|^2)^s\,|f_ n(w)|^p \int_{D}
\frac{d\mu(z)}{|1-\bar{w}z|^{2s}}
\\
&\le C |\|f_ n|\|_{F(p,p-2,s)}\sup_{|w|>r_ 0}\left (\log
\frac{2}{1-|w|^2}\right)^p\int_{D}
\frac{(1-|w|^2)^s}{|1-\bar{w}z|^{2s}}\,d\mu(z)
\\
&<C\varepsilon.
\end{split}
\end{displaymath}
Now it is easy to deduce that
$$\lim_{n\rightarrow \infty}\sup_{|w|> r_ 0} A_ 2(n,w)=0$$
completing this part of the proof.\\

Conversely, suppose that
for any bounded sequence $\{f_n\}$ in $F(p,p-2,s)$
with $f_n(z)\to 0$ uniformly on every compact subset of $D$ we have
$\|f_n\|_{T^{\infty}_{p,s}(\mu)}\to0$.
Let $\{I_n\}$ be a sequence of subarcs of $\partial D$
such that $|I_n|\to0$.
Let $\zeta_n$ be the center of $I_n$, $w_n=(1-|I_n|)\zeta_n$,
and
$$
f_n(z)=\left(\log\frac2{1-|w_n|^2}\right)^{-1}\left(\log\frac2{1-\bar w_nz}\right)^2.
$$
Arguing as in Lemma~\ref{8}, we easily see that $\{f_n\}$ is a bounded sequence on $F(p,p-2,s)$,
and $f_n(z)\to 0$ uniformly on every compact subset of $D$.
Thus
\begin{displaymath}
\begin{split}
\frac1{|I_n|^s}\left(\log\frac2{|I_n|}\right)^p\mu(S(I_n))
\le C \,\frac1{|I_n|^s}\int_{S(I_n)}|f_n(z)|^p\,d\mu(z) \le C\,\|f_n\|_{T^{\infty}_{p,s}(\mu)}^p \rightarrow 0,
\end{split}
\end{displaymath}
proving that $\mu$ is a vanishing $(p,s)$-logarithmic Carleson measure.
The proof is complete. \hspace*{\fill}
$\square$\\

Using Lemma~\ref{1} we immediately get the following corollary of Theorem~\ref{2}.

\begin{coro}\label{11}
Let $s>0$ and $p>0$ satisfy $s+p>1$, and let $\mu$ be a
nonnegative Borel measure on $D$. Suppose that the point
evaluation is a bounded functional on $T^{\infty}_{p,s}(\mu)$.
Then $I:F(p,p-2,s)\to T^{\infty}_{p,s}(\mu)$ is compact if and
only if $\mu$ is a vanishing $(p,s)$-logarithmic Carleson measure.
\end{coro}

\begin{section}
{Embeddings from $F(p,p\alpha-2,s)$ into $T^{\infty}_{p,s}(\mu)$ with $\alpha\neq1$}
\end{section}

In the previous section we studied the embedding $I:F(p,p-2,s)\to
T^{\infty}_{p,s}(\mu)$. In this section we consider the embedding
$I:F(p,p\alpha-2,s)\to T^{\infty}_{p,s}(\mu)$ when $\alpha\neq 1$.
First, we look at the case $0<\alpha<1$. We need first the
following lemma from \cite{osz}.

\begin{lemma}\label{12}
Let $0<\alpha<1$ and let $T$ be a bounded linear operator from $B^{\alpha}$
into a normed linear space $Y$.
Then: $T$ is compact if and only if
$\|Tf_{n}\|_{Y}\to 0$, whenever $\{f_{n}\}$ is a bounded
sequence in $B^{\alpha}$ that converges to $0$ uniformly on $\overline{D}$, the closure of $D$.
\end{lemma}
\textit{Remark.} It is clear from the proof given in \cite{osz}
that the above lemma also holds for $Y=T^{\infty}_{p,s}(\mu)$ in
the case $0<p<1$.
\begin{theorem}\label{13}
Let $0<\alpha<1$, $s>0$ and $p>0$ satisfy $s+p\alpha>1$, and let
$\mu$ be a nonnegative Borel measure on $D$. Then the following
conditions are equivalent:
\begin{enumerate}
\item[(i)] $I:F(p,p\alpha-2,s)\to T^{\infty}_{p,s}(\mu)$ is bounded.
\item[(ii)] $I:F(p,p\alpha-2,s)\to T^{\infty}_{p,s}(\mu)$ is compact.
\item[(iii)] $\mu$ is an $s$-Carleson measure.
\end{enumerate}
\end{theorem}

\begin{proof}
(i)$\Longrightarrow$(iii).
Suppose $I:F(p,p\alpha-2,s)\to T^{\infty}_{p,s}(\mu)$
is a bounded operator. Let $f(z)=1$. Then obviously $f\in F(p, p\alpha-2,s)$.
Hence,
$$
\frac1{|I|^s}\mu(S(I)) =\frac1{|I|^s}\int_{S(I)}|1|^p\,d\mu(z) \le
C|\|1\||_{F(p,p\alpha-2,s)}^p\le C.
$$
Thus $\mu$ is an $s$-Carleson measure.

(iii)$\Longrightarrow$(i). Suppose that $\mu$ is an $s$-Carleson
measure. By Corollary 2.8 in \cite{zhao1},
$F(p,p\alpha-2,s)\subset B^{\alpha}$. If we can prove that $I:
B^{\alpha}\to T^{\infty}_{p,s}(\mu)$ is bounded then we are done.
Since $B^{\alpha}\subset H^{\infty}$ for $0<\alpha<1$, a standard
application of the Closed Graph Theorem shows that
$\|f\|_{H^{\infty}}\le C|\|f\||_{B^{\alpha}}$. Hence,
$$
\frac1{|I|^s}\int_{S(I)}|f(z)|^p\,d\mu(z) \le
\|f\|_{H^{\infty}}^p\frac1{|I|^s}\mu(S(I)) \le
C|\|f\||_{B^{\alpha}}.
$$
Thus $I: B^{\alpha}\to T^{\infty}_{p,s}(\mu)$ is bounded.
Consequently, $I:F(p,p\alpha-2,s)\to T^{\infty}_{p,s}(\mu)$ is bounded.

(iii)$\Longrightarrow$(ii). We further prove that (iii) implies
(ii). Since $F(p,p\alpha-2,s)\subset B^{\alpha}$ (see Corollary
2.8 in \cite{zhao1}), it is enough to prove that $I:B^\alpha\to
T^{\infty}_{p,s}(\mu)$ is compact. Since we just proved that $I:
B^{\alpha}\to T^{\infty}_{p,s}(\mu)$ is bounded, by
Lemma~\ref{12}, we need only prove that
$\|f_{n}\|_{T^{\infty}_{p,s}(\mu)}\to 0$, whenever $\{f_{n}\}$ is
a bounded sequence in $B^{\alpha}$ that converges to $0$ uniformly
on $\overline{D}$. Since $\{f_n\}$ converges to $0$ uniformly on
$\overline{D}$. for any given $\e>0$, there exists a positive
integer $N$ such that $|f_n(z)|<\e$ for any $n>N$. Hence, for any
$n>N$ we have
$$
\|f_n\|_{T^{\infty}_{p,s}(\mu)}^p
=\sup_{I\subset\partial D}\frac{1}{|I|^s}\int_{S(I)}|f_n(z)|^p\,d\mu(z)
\le \e\sup_{I\subset\partial D}\frac{1}{|I|^s}\int_{S(I)}\,d\mu(z)
\le C\e.
$$
Therefore $\lim_{n\to\infty}\|f_n\|_{T^{\infty}_{p,s}(\mu)}^p=0$,
and so $I:F(p,p\alpha-2,s)\to T^{\infty}_{p,s}(\mu)$ is compact.

(ii)$\Longrightarrow$(i). This is obvious.
The proof is complete.
\end{proof}

Next, we consider the case $\alpha>1$.
\begin{theorem}\label{14a}
 Let $p>0$, $s>0$ and $\alpha>1$ with $s+p\alpha>1$.
Let $\mu$ be a finite positive Borel measure on $D$. If $\mu$ is a Carleson measure for $D_{s+p\alpha-2}^p$,
then the embedding $I:F(p,p\alpha-2,s)\rightarrow
T_{p,s}^{\infty}(\mu)$ is bounded.
\end{theorem}

\begin{proof}
 Let $\beta=s+p(\alpha-1)$ so that $s+p\alpha-2=\beta+p-2$. We
prove first that $\mu$ must be a  $\beta$-Carleson measure. For
each $a\in D$, consider the test functions
$$ f_ a(z)=\frac{(1-|a|^2)^{\frac{\beta}{ p}}}{(1-\bar{a}z)^{\frac{2\beta}{p}}},\qquad z\in D.$$
Then $\displaystyle{\sup_{a\in D}\|f_ a\|_{D^p_{\beta+p-2}}\leq
C}$, and since $\mu$ is a Carleson measure for $D^p_{\beta+p-2}$,
one obtains
\begin{displaymath}
\int_{D}\frac{(1-|a|^2)^{\beta}}
{|1-\bar{a}z|^{2\beta}}\,d\mu(z)=\int_{D}|f_ a(z)|^p\,d\mu(z)\leq C.
\end{displaymath}
Hence $\mu$ is a $\beta$-Carleson measure. Now, by Proposition \ref{16},
we have that $\|f\|_{T_{p,s}^{\infty}(\mu)}^p$ is comparable to the quantity
\begin{displaymath}
\sup_{a\in D}\int_{D} \frac{(1-|a|^2)^t}{|1-\bar{a}z|^{s+t}}|f(z)|^p\,d\mu(z)
\end{displaymath}
for all $t>0$. Fix $a\in D$, and let $t=s+2p(\alpha-1)$, and $f\in F(p,p\alpha-2,s)$. Then
\begin{displaymath}
\int_{D} \frac{(1-|a|^2)^t}{|1-\bar{a}z|^{s+t}}|f(z)|^p\,d\mu(z)\leq \max(1,2^{p-1})\big(I_ 1(a)+I_ 2(a)\big),
\end{displaymath}
where
\begin{displaymath}
I_ 1(a)=|f(a)|^p\int_{D} \frac{(1-|a|^2)^t}{|1-\bar{a}z|^{s+t}}\,d\mu(z),
\end{displaymath}
and
\begin{displaymath}
I_ 2(a)=\int_{D} \frac{(1-|a|^2)^t}{|1-\bar{a}z|^{s+t}}\,|f(z)-f(a)|^p\,d\mu(z).
\end{displaymath}
Since $F(p,p\alpha-2,s)\subset B^{\alpha}$, and $\alpha>1$, one has
\begin{equation}\label{E1}
|f(a)|^p \leq C |\|f\||_{F(p,p\alpha-2,s)}^p(1-|a|^2)^{-p(\alpha -1)}.
\end{equation}
This, together with the fact that $\mu$ is an
$\big[s+p(\alpha-1)\big]$-Carleson measure and Lemma \ref{4b},
gives
\begin{displaymath}
\begin{split}
I_ 1(a)&\leq C |\|f\||_{F(p,p\alpha-2,s)}^p(1-|a|^2)^{t-p(\alpha -1)}\int_{D}\frac{d\mu(z)}{|1-\bar{a}z|^{s+t}}
\\
&=C |\|f\||_{F(p,p\alpha-2,s)}^p(1-|a|^2)^{s+p(\alpha -1)}\int_{D}\frac{d\mu(z)}{|1-\bar{a}z|^{2s+2p(\alpha-1)}}
\\
&\leq C |\|f\||_{F(p,p\alpha-2,s)}^p.
\end{split}
\end{displaymath}
It remains to deal with the term $I_ 2(a)$. Set
$$f_ a(z)=\frac{f(z)-f(a)}{(1-\bar{a}z)^{\frac{s+t}{p}}}.$$
Since $\mu$ is a Carleson measure for $D^p_{s+p\alpha-2}$, we obtain
\begin{displaymath}
\begin{split}
I_ 2(a)=&(1-|a|^2)^t \int_{D}|f_ a(z)|^p \,d\mu(z)
\\
\leq & C (1-|a|^2)^t |f_ a(0)|^p+C (1-|a|^2)^t\int_{D} \big |(f_ a)'(z)\big |^p (1-|z|^2)^{s+p\alpha-2}\,dA(z).
\end{split}
\end{displaymath}
By \eqref{E1}, one has
\begin{displaymath}
\begin{split}
(1-|a|^2)^t |f_ a(0)|^p &=(1-|a|^2)^t |f(a)-f(0)|^p
\\
&\le C (1-|a|^2)^{s+p(\alpha-1)}\|f\|^p_{F(p,p\alpha-2,s)}
\\
&\leq C \|f\|^p_{F(p,p\alpha-2,s)}.
\end{split}
\end{displaymath}
On the other hand, by Lemma \ref{4b} and Proposition \ref{10},
\begin{displaymath}
\begin{split}
(1-|a|^2)^t&\int_{D}\big |(f_ a)'(z)\big |^p (1-|z|^2)^{s+p\alpha-2}\,dA(z)
\\
&\leq C(1-|a|^2)^t\int_{D}|f'(z)|^p\frac{(1-|z|^2)^{s+p\alpha-2}}{|1-\bar{a}z|^{s+t}}dA(z)
\\
& \,\, + C(1-|a|^2)^t \int_{D} \frac{|f(z)-f(a)|^{p}}{|1-\bar{a}z|^{s+t+p}}(1-|z|^2)^{s+p\alpha-2}dA(z)
\\
&\leq C \|f\|^p_{F(p,p\alpha-2,s)}.
\end{split}
\end{displaymath}
All together gives
$$
I_ 2(a)\leq C \|f\|^p_{F(p,p\alpha-2,s)},
$$
finishing the proof of the
theorem.
\end{proof}

\begin{theorem}\label{14}
Let $\alpha>1$, $s>0$, and $p>0$ with $s+p\alpha>1$, and let $\mu$
be a nonnegative Borel measure on $D$. If $0<p\le 1$; or $p>1$ with $s+p(\alpha -1)>1$;
 or $1<p\le 2$ with $s+p(\alpha -1)=1$, then
\begin{enumerate}
\item[(i)]
The identity operator $I:F(p,p\alpha-2,s)\to T^{\infty}_{p,s}(\mu)$
is bounded if and only if $\mu$ is an $[s+p(\alpha-1)]$-Carleson measure.
\item[(ii)] The following two conditions are equivalent:
\begin{enumerate}
\item[(a)]
For any bounded sequence $\{f_n\}$ in $F(p,p\alpha-2,s)$ satisfying $f_n(z)\to0$ uniformly
on every compact subset of $D$,
$$
\lim_{n\to\infty}\|f_n\|_{T_{p,s}^{\infty}(\mu)}=0.
$$
\item[(b)] $\mu$ is a vanishing $[s+p(\alpha-1)]$-Carleson measure.
\end{enumerate}
\end{enumerate}
\end{theorem}

\begin{proof}
(i) Suppose first that $\mu$ is an $[s+p(\alpha-1)]$-Carleson measure. In all the cases considered,
 this is equivalent to $\mu$ being a Carleson measure for $D^p_{s+\alpha p-2}$ (see \cite{wu}.
Note that for $0<p\le 1$, the implication needed here is proved in
our proof of Lemma \ref{5}). Thus, by Theorem \ref{14a} it follows that
$I:F(p,p\alpha-2,s) \to T^{\infty}_{p,s}(\mu)$ is bounded.
\par
Conversely, let $I:F(p,p\alpha-2,s)\to T^{\infty}_{p,s}(\mu)$ be
bounded. Given any arc $I\subset \partial D$, let
$a=(1-|I|)\zeta$, where $\zeta$ is the center of $I$. Let
$$
f_a(z)=(1-\bar az)^{1-\alpha}.
$$
By Lemma~\ref{8}, there is a constant $C>0$ such that $
|\|f_a\||_{F(p,p\alpha-2,s)}\le C. $ Since $I:F(p,p\alpha-2,s)\to
T^{\infty}_{p,s}(\mu)$ is bounded, we get
$$
\frac1{|I|^s}\int_{S(I)}|f_a(z)|^p\,d\mu(z)\le C|\|f_a\||_{F(p,p\alpha-2,s)}^p\le C.
$$
It is easy to see that
$$
f_a(z)\approx |I|^{1-\alpha}
$$
for any $z\in S(I)$. Combining the above two inequalities we get
$$
\frac1{|I|^{s+p(\alpha-1)}}\mu(S(I))\le C.
$$
Thus $\mu$ is an $[s+p(\alpha-1)]$-Carleson measure.

Next, we prove (ii). Let $\mu$ be a vanishing
$[s+p(\alpha-1)]$-Carleson measure and $\{f_n\}$ be a bounded sequence in $F(p,p\alpha-2,s)$
with $f_n(z)\to 0$ uniformly on every compact subset of $D$. We first deal with the cases considered with $s+p(\alpha-1)\ge 1$.
Let $D_r=\{z\in D:\,|z|\le r\}$, and $\mu_r=\mu|_{D\setminus D_r}$ be the restriction of $\mu$ on
$D\setminus D_r$. By Proposition \ref{pn}, if $r\to1$ then
$$
\sup_{I\subset\partial D}\frac{\mu_r\big(S(I)\big)}{|I|^{s+p(\alpha-1)}}
\to0.
$$
Recall that, in the introduction we have seen that,
when $\alpha>1$, if $g\in B^{\alpha}$ then there is a constant $C>0$,
independent of $g$,
such that
$$
|g(z)|\le C\||g\||_{B^{\alpha}}(1-|z|^2)^{1-\alpha}.
$$
By Corollary 2.8 in \cite{zhao1}, $F(p,p\alpha-2,s)\subset B^{\alpha}$.
Thus, if $f\in F(p,p\alpha-2,s)$ then
$$
|f(z)|\le C\||f\||_{B^{\alpha}}(1-|z|^2)^{1-\alpha}
\le C\||f\||_{F(p,p\alpha-2,s)}(1-|z|^2)^{1-\alpha}.
$$
Hence, given any arc $I\subset \partial D$
with $w=(1-|I|)\zeta$ and $\zeta$ is the center of $I$,
for our sequence $\{f_n\}$ we can find a uniform constant $C>0$
such that
$$
|f_n(w)|\le \frac{C\||f_n\||_{F(p,p\alpha-2,s)}}{(1-|w|^2)^{\alpha-1}}
=\frac{C\||f_n\||_{F(p,p\alpha-2,s)}}{|I|^{\alpha-1}}.
$$
The fact that the limit $\|f_n\|_{T^{\infty}_{p,s}}(\mu)\to0$ now follows from part (i) and
$$
\int_{S(I)}|f_n(z)|^p\,d\mu(z)
\le C\int_{S(I)}|f_n(z)|^p\,d\mu_{D_r}(z)
+C\||f_n\||_{F(p,p\alpha-2,s)}^p\frac{\mu_r(S(I))}{{|I|^{p(\alpha-1)}}}.$$
The proof for $0<p\le 1$ follows a similar argument as the corresponding one in Theorem \ref{2}. We left the details to the interested reader.\\

Conversely, suppose that for any bounded sequence $\{f_n\}$ in $F(p,p\alpha-2,s)$
with $f_n(z)\to 0$ uniformly on every compact subset of $D$ we have
$\|f_n\|_{T^{\infty}_{p,s}}(\mu)\to0$.
Let $\{I_n\}$ be a sequence of subarcs of $\partial D$
such that $|I_n|\to0$.
Let $\zeta_n$ be the center of $I_n$, $w_n=(1-|I_n|)\zeta_n$,
and
$$
f_n(z)=\frac{1-|w_ n|^2}{(1-\bar w_nz)^{\alpha}}.
$$
By a proof similar to that of Lemma~\ref{8},
we can easily see that $\{f_n\}$ is a bounded sequence on $F(p,p\alpha-2,s)$,
and $f_n(z)\to 0$ uniformly on every compact subset of $D$.
Clearly, for any $z\in S(I)$,
$$
|f_n(z)|\approx |I_n|^{1-\alpha}.
$$
Thus
$$
\frac{\mu(S(I_n))}{|I_n|^{s+p(\alpha-1)}}
\le C \,\frac1{|I_n|^s}\int_{S(I_n)}|f_n(z)|^p\,d\mu(z)\le C\,\|f_n\|_{T^{\infty}_{p,s}(\mu)}^p \rightarrow 0.
$$
Hence $\mu$ is a vanishing $[s+p(\alpha-1)]$-Carleson measure.
The proof is complete.
\end{proof}

Using Lemma~\ref{1} we immediately get the following corollary of Theorem~\ref{14}.

\begin{coro}\label{14-a}
Let $\alpha>1$, $s>0$ and $p>0$ satisfy $s+\alpha p>1$, and let $\mu$
be a nonnegative Borel measure on $D$.
Suppose that the point evaluation is a bounded functional on $T^{\infty}_{p,s}(\mu)$. If $0<p\le 1$; or $p>1$ with $s+p(\alpha -1)>1$;
 or $1<p\le 2$ with $s+p(\alpha -1)=1$, then
$I:F(p,p\alpha-2,s)\to T^{\infty}_{p,s}(\mu)$ is compact if and only if
$\mu$ is a vanishing $[s+p(\alpha-1)]$-Carleson measure.
\end{coro}

\begin{section}
{Riemann-Stieltjes operators on $F(p,q,s)$}
\end{section}

In this section we look at applications of our main theorems
to Riemann-Stieltjes integral operators.
Recall that, for $g\in H(D)$, the Riemann-Stieltjes integral operator
$J_g$ is defined by
$$
J_gf(z)=\int_0^zf(\zeta)g'(\zeta)\,d\zeta
$$
for $f\in H(D)$. The operators $J_g$ were first used by Ch. Pommerenke in
\cite{pommerenke} to characterize BMOA functions. They were first
systematically studied by A. Aleman and A. G. Siskakis  in
\cite{as1}. They proved that $J_g$ is bounded on the Hardy space
$H^p$ if and only if $g\in BMOA$. Thereafter there have been many
works on these operators. See,  \cite{ac}, \cite{as2}, \cite{hu},
\cite{pp2}, and \cite{sz} for a few examples. Here
we are considering boundedness and compactness of these operators
on $F(p,q,s)$.

For $0<p<\infty$, $-2<q<\infty$, $0\le s<\infty$ such that $q+s>-1$,
we define the space $F_L(p,q,s)$, called \textit{logarithmic $F(p,q,s)$ space},
as the space of analytic functions $f$ on $D$ satisfying
$$
\sup_{a\in D}\left(\log\frac{2}{1-|a|^2}\right)^p
\int_D|f'(z)|^p(1-|z|^2)^q(1-|\p_a(z)|^2)^s\,dA(z)<\infty.
$$
We also say $f\in F_{L,0}(p,q,s)$, if
$$
\lim_{|a|\to0}\left(\log\frac{2}{1-|a|^2}\right)^p
\int_D|f'(z)|^p(1-|z|^2)^q(1-|\p_a(z)|^2)^s\,dA(z)=0.
$$
The following result establishes the boundedness of $J_g$ on $F(p,q,s)$.

\begin{theorem}\label{17}
Let $\alpha>0$, $s>0$ and $p>0$ satisfy $s+p\alpha>1$.
Then we have the following results.
\begin{enumerate}
\item[(i)] As $0<\alpha<1$, $J_g$ is bounded on $F(p,p\alpha-2,s)$ if and only if
$g\in F(p,p\alpha-2,s)$.
\item[(ii)]
$J_g$ is bounded on $F(p,p-2,s)$ if and only if $g\in F_L(p,p-2,s)$.
\item[(iii)] Let $\alpha>1$ and $\gamma=s+p(\alpha-1)$. For the cases $0<p\le 1$; or $p>1$ and $\gamma>1$; or $1<p\le 2$ and $\gamma=1$, we have that
$J_g$ is bounded on $F(p,p\alpha-2,s)$ if and only if
$g\in F(p,p-2,\gamma)$.
\end{enumerate}
\end{theorem}

\begin{proof}
Since $(J_gf)'=fg'$, the operator $J_g$ is bounded on $F(p,p\alpha-2,s)$
if and only if
$$
\int_D|f(z)|^p|g'(z)|^p(1-|z|^2)^{p\alpha-2}(1-|\p_a(z)|^2)^s\,dA(z)
\le C\||f\||_{F(p,p\alpha-2,s)}.
$$
Since $1-|\p_a(z)|^2=(1-|z|^2)|\p_a'(z)|$, the above inequality is equivalent to
\begin{equation}\label{eq6}
\int_D|f(z)|^p|\p'_a(z)|^s\,d\mu_g(z)\le C\||f\||_{F(p,p\alpha-2,s)},
\end{equation}
where
$$
d\mu_g(z)=|g'(z)|^p(1-|z|^2)^{s+p\alpha-2}\,dA(z).
$$
By Proposition~\ref{16}, (\ref{eq6}) means that
$I:\,F(p,p\alpha-2,s)\to T^{\infty}_{p,s}(\mu_g)$
is a bounded operator.

(i) As $0<\alpha<1$, by Theorem~\ref{13}, (\ref{eq6}) is equivalent to
that $\mu_g$ is an $s$-Carleson measure.
By Lemma~\ref{4b}, this means that
$$
\sup_{a\in D}\int_D|\p_a'(z)|^s\,d\mu_g(z)<\infty,
$$
which is the same as
$$
\sup_{a\in D}\int_D|g'(z)|^p(1-|z|^2)^{p\alpha-2}(1-|\p_a(z)|^2)^s\,dA(z)<\infty.
$$
Thus $g\in F(p,p\alpha-2,s)$.

(ii) As $\alpha=1$, by Theorem~\ref{2}, (\ref{eq6}) is equivalent to
that $\mu_g$ is a $(p,s)$-logarithmic Carleson measure.
By Lemma \ref{4b} (or also by Theorem 2 in \cite{zhao3}), this is equivalent to
$$
\sup_{a\in D}\left(\log\frac 2{1-|a|}\right)^p\int_D|\p'_a(z)|^s\,d\mu_g(z)<\infty,
$$
or
$$
\sup_{a\in D}\left(\log\frac 2{1-|a|}\right)^p\int_D|g'(z)|^p
(1-|z|^2)^{p-2}(1-|\p_a(z)|^2)^s\,dA(z)<\infty,
$$
which means $g\in F_L(p,p-2,s)$.

(iii). Let $\alpha>1$. By Theorem~\ref{14}, in all the cases considered,
(\ref{eq6}) means
that $\mu_g$ is an $[s+p(\alpha-1)]$-Carleson measure.
By Lemma~\ref{4b}, this is equivalent to
$$
\sup_{a\in D}\int_D|\p'_a(z)|^{s+p(\alpha-1)}\,d\mu_g(z)<\infty,
$$
or
$$
\sup_{a\in D}\int_D|g'(z)|^p(1-|z|^2)^{p-2}(1-|\p_a(z)|^2)^{s+p(\alpha-1)}\,dA(z)<\infty.
$$
which means $g\in F(p,p-2,s+p(\alpha-1))$.  The proof is complete.
\end{proof}

\noindent \textit{Remark.} Parts of the above results have been
proved before. When $p=2$, $\alpha=1$ and $s=1$, it is known that
$F(2,0,1)=BMOA$, the space of analytic functions of bounded mean
oscillation. In this case, the above result was proved by Siskakis
and the second author in \cite{sz}. When $p=2$, $\alpha=1$,
$0<s<1$, we know that $F(p,p\alpha-2,s)=Q_s$. In this case, the
result was proved by Xiao in \cite{xiao2}. Note also that in the
 case $\alpha>1$ and $\gamma=s+p(\alpha-1)>1$,
by Theorem 1 in \cite{zhao2} (also see Theorem 1.3 in \cite{zhao1}),
we know that $F(p,p-2,\gamma)=B$, the Bloch space.
\medskip

We can also use our previous results to characterize compactness
of $J_g$ on $F(p,q,s)$.

\begin{theorem}\label{18}
Let $\alpha>0$, $s>0$ and $p>0$ satisfy $s+p\alpha>1$.
Then we have the following results.
\begin{enumerate}
\item[(i)] As $0<\alpha<1$, $J_g$ is compact on $F(p,p\alpha-2,s)$ if and only if
$g\in F(p,p\alpha-2,s)$.
\item[(ii)]
$J_g$ is compact on $F(p,p-2,s)$ if and only if
$g\in F_{L,0}(p,p-2,s)$.
\item[(iii)] Let $\alpha>1$ and $\gamma=s+p(\alpha-1)$. For the cases $0<p\le 1$; or $p>1$ and $\gamma>1$; or $1<p\le 2$ and $\gamma=1$, we have that
$J_g$ is compact on $F(p,p\alpha-2,s)$ if and only if
$g\in F_0(p,p-2,\gamma)$.
\end{enumerate}
\end{theorem}

\begin{proof}
As in the proof of Theorem~\ref{17}, we know that the operator
$J_g$ is compact on $F(p,p\alpha-2,s)$ if and only if for any
bounded sequence $\{f_n\}$ in $F(p,p\alpha-2,s)$ with $f_n(z)\to0$
uniformly on compact subsets of $D$, we have
$$
\lim_{n\to\infty}\|J_g(f_n)\|_{F(p,p\alpha-2,s)}=0.
$$
Thus we can apply Theorem~\ref{2}, Theorem~\ref{13} and Theorem~\ref{14}
to complete the proof.
The proof goes in the same way as the proof of Theorem~\ref{17},
so we omit the details here.
\end{proof}

\textit{Remark.} Recall that a function $f$ analytic on $D$ belongs to $F_ 0(p,q,s)$ if $f\in F(p,q,s)$ and
 $$\lim_{|a|\rightarrow 1^{-}} \int_{D}f'(z)|^p(1-|z|^2)^q(1-|\p_a(z)|^2)^s\,dA(z)=0 .$$
 When $\alpha>1$ and $\gamma=s+p(\alpha-1)>1$, one has $F_ 0(p,p-2,\gamma)=B_ 0$, the little Bloch space.
\begin{section}
{Pointwise multiplication operators on $F(p,q,s)$}
\end{section}

For an $F$-space $X$ of analytic functions,
we denote by $M(X)$ the space of all \textit{pointwise multipliers}
on $X$, that is,
$$
M(X)=\{g\in H(D): \, fg\in X \textrm{ for all }f\in X\}.
$$
Let $M_g$ denote the pointwise
multiplication operator, given by  $M_gf=fg$.
Since the Closed Graph Theorem is still available for $F$-spaces (see \cite[Chap. II]{ds}), we know that $g\in M(X)$ if and only if
$M_g$ is bounded on $X$. We also denote
$$
M_0(X)=\{g\in H(D):\, M_g\textrm{ is compact on } X\}.
$$

In this section, we give characterizations of $M(F(p,q,s))$ and $M_0(F(p,q,s))$.
To this end, we first study another integral operator. Let $g$ be an analytic function
on $D$. Define
$$
I_gf(z)=\int_0^zf'(t)g(t)\,dt,\qquad f\in H(D).$$

\begin{propo}\label{19}
Let $\alpha>0$, $s>0$ and $p>0$ satisfy $s+p\alpha>1$.
Then
\begin{enumerate}
\item[(i)] $I_g$ is bounded on $F(p,p\alpha-2,s)$ if and only if $g\in H^{\infty}$.
\item[(ii)] $I_g$ is compact on $F(p,p\alpha-2,s)$ if and only if $g=0$.
\end{enumerate}
\end{propo}

\begin{proof}
(i) Let $g\in H^{\infty}$. Since $(I_gf)'=f'g$, it is obvious that $I_g$ is bounded on
$F(p,p\alpha-2,s)$.

Conversely,
suppose $I_g$ is bounded on $F(p,p\alpha-2,s)$.
Let $a\in D$, and let
$$
h_a(z)=\frac{1-|a|^2}{\alpha(1-\bar az)^{\alpha}}.
$$
By a computation similar to the proof of Lemma~\ref{8},
we can see that
$$
\sup_{a\in D}\|h_a\|_{F(p,p\alpha-2,s)}<\infty.
$$
Since $F(p,p\alpha-2,s)\subset B^{\alpha}$, we know that
$I_g$ is bounded from $F(p,p\alpha-2,s)$ to $B^{\alpha}$.
Thus,
$$
\|I_gh_a\|_{B^{\alpha}}\le C\|h_a\|_{F(p,p\alpha-2,s)}\le C.
$$
Since $(I_gh_a)'=h_a'g$, the above inequality is the same as
\begin{equation}\label{eq7}
\sup_{z\in D}|h_a'(z)||g(z)|(1-|z|^2)^{\alpha}\le C.
\end{equation}
Since $h_a'(z)=\bar a(1-|a|^2)(1-\bar az)^{-\alpha-1}$,
we get $h'_a(a)=\bar a(1-|a|^2)^{-\alpha}$.
Letting $z=a$ in (\ref{eq7}) we get
$$
|a||g(a)|\le C.
$$
Since $g$ is analytic on $D$, this implies that $g\in H^{\infty}$.

(ii) Let $g=0$, then trivially, $I_g$ is compact on $F(p,p\alpha-2,s)$.

Conversely, suppose $I_g$ is compact on $F(p,p\alpha-2,s)$.
Then $I_g$ is bounded on $F(p,p\alpha-2,s)$, and so by (i),
$g\in H^{\infty}$.
Since $F(p,p\alpha-2,s)\subset B^{\alpha}$, we know that
$I_g$ is compact from $F(p,p\alpha-2,s)$ to $B^{\alpha}$.
Let $\{a_n\}$ be any sequence of points in $D$ such that
$|a_n|\to1$, and let
$$
h_n(z)=\frac{1-|a_n|^2}{\alpha(1-\bar a_nz)^{\alpha}}.
$$
As in the proof of (i), we know that
$$
\sup_{n\ge 1}\|h_n\|_{F(p,p\alpha-2,s)}<\infty.
$$
Obviously, ${h_n}$ converges to $0$ uniformly on compact subsets of $D$.
Hence
$$
\lim_{n\to\infty}\|I_gh_n\|_{B^{\alpha}}=0,
$$
which is equivalent to
$$
\lim_{n\to\infty}\sup_{z\in D}|h_n'(z)||g(z)|(1-|z|^2)^{\alpha}=0.
$$
Let $z=a_n$ we see
$$
\lim_{n\to\infty}|h_n'(a_n)||g(a_n)|(1-|a_n|^2)^{\alpha}=0.
$$
Since $h_n'(a_n)=\bar a_n(1-|a_n|^2)^{-\alpha}$, the above equation becomes
$$
\lim_{n\to\infty}|a_n||g(a_n)|=0.
$$
Since $g\in H^{\infty}$, we must have $g=0$.
The proof is complete.
\end{proof}

\begin{theorem}\label{20}
Let $\alpha>0$, $s>0$ and $p>0$ satisfy $s+p\alpha>1$.
Then we have the following results.
\begin{enumerate}
\item[(i)] If $0<\alpha<1$, then
$
M(F(p,p\alpha-2,s))=F(p,p\alpha-2,s).
$
\item[(ii)] As $\alpha=1$, we have that
$$
M(F(p,p-2,s))=H^{\infty}\cap F_L(p,p-2,s).
$$
\item[(iii)] Let $\alpha>1$ and $\gamma=s+p(\alpha-1)$. Then
\begin{enumerate}
\item If $0<p\le 1$, then $M(F(p,p\alpha-2,s))=F(p,p-2,\gamma)$.

\item If $p>1$ and $\gamma>1$, then
$
M(F(p,p\alpha-2,s))=H^{\infty}\cap B=H^{\infty}.
$

\item If $1<p\le 2$ and $\gamma=1$, then $M(F(p,p\alpha-2,s))=F(p,p-2,\gamma)\cap H^{\infty}$.
\end{enumerate}
\end{enumerate}
\end{theorem}

\begin{proof}
(i) If $0<\alpha<1$, then $F(p,\alpha p-2,s)\subset H^{\infty}$, and it is easy to see that the space $F(p,\alpha p-2,s)$ is an algebra.

In order to prove (ii) and (iii), notice that
\begin{equation}\label{eq8}
M_gf(z)=f(z)g(z)=f(0)g(0)+I_gf(z)+J_gf(z).
\end{equation}
Hence, if $I_g$ and $J_g$ are both bounded on $F(p,p\alpha-2,s)$, then
$M_g$ is also bounded on $F(p,p\alpha-2,s)$.
So the inclusions
$$
H^{\infty}\cap F_L(p,p-2,s)\subset M(F(p,p-2,s)), \textrm{ for }\alpha=1,
$$
and, in all the cases considered,
$$
H^{\infty}\cap F(p,p-2,\gamma)\subset M(F(p,p\alpha-2,s)), \textrm{ for }\alpha>1,
$$
are direct consequences of Theorem~\ref{17} and (i) of Proposition~\ref{19}.

Now we are proving the inverse inclusions.
Let $\alpha>0$ and $g\in M(F(p,p\alpha-2,s))$.
For $a\in D$, let
$$
\psi_a(z)=\frac{(1-|a|^2)^2}{(1-\bar az)^{\alpha+1}}-\frac{1-|a|^2}{(1-\bar az)^{\alpha}}.
$$
By a computation similar to the proof of Lemma~\ref{8}, we get that
$$
\sup_{a\in D}\|\psi_a\|_{F(p,p\alpha-2,s)}\le C<\infty.
$$
It is also easy to check that
$$
\psi_a(a)=0, \qquad \psi_a'(a)=\bar a(1-|a|^2)^{-\alpha}.
$$
Since $M_g$ is bounded on $F(p,p\alpha-2,s)$ and $F(p,p\alpha-2,s)\subset B^{\alpha}$,
we get that $M_g$ is bounded from $F(p,p\alpha-s,s)$ to $B^{\alpha}$.
Hence, there is a constant $C>0$ such that
\begin{eqnarray*}
C\|\psi_a\|_{F(p,p\alpha-2,s)}
&>&\sup_{a\in D}\|M_g\psi_a\|_{B^{\alpha}}\\
&=&\sup_{a\in D}\sup_{z\in D}|g'(z)\psi_a(z)+g(z)\psi_a'(z)|(1-|z|^2)^{\alpha}\\
&\ge&|g'(a)\psi_a(a)+g(a)\psi_a'(a)|(1-|a|^2)^{\alpha}\\
&=&|g(a)||a|(1-|a|^2)^{-\alpha}(1-|a|^2)^{\alpha}\\
&=&|g(a)||a|.
\end{eqnarray*}
Hence $g\in H^{\infty}$. By Proposition~\ref{19}, $I_g$ is bounded on $F(p,p\alpha-2,s)$.
Hence, from (\ref{eq8}) we know that $J_g$ is also bounded on $F(p,p\alpha-2,s)$.
The results now follow from Theorem~\ref{16} and the fact that $g\in H^{\infty}$. When $\alpha>1$, in all the cases considered, we obtain $M(F(p,\alpha p-2,s))=H^{\infty}\cap F(p,p-2,\gamma)$, and the statements of (iii) are consequences of $F(p,p-2,\gamma)\subset H^{\infty}$ for $0<p\le 1$, and the fact that $H^{\infty}\subset B$ and $F(p,p-2,\gamma)=B$ for $\gamma>1$.
The proof is complete.
\end{proof}

\noindent
\textit{Remark.} The above result for $BMOA=F(2,0,1)$
was obtained first by Stegenga in \cite{stegenga1}
(see also \cite{of} and \cite{zhao3}).
When $p=2$, $\alpha=1$, $0<s<1$, we know $F(2,0,s)=Q_s$ with $0<s<1$.
In this case the result was proved by
Pau and Pel\'aez in \cite{pp}.
\medskip

\begin{theorem}\label{21}
Let $\alpha>0$, $s>0$ and $p>0$ satisfy $s+p\alpha>1$.
Then $M_0(F(p,p\alpha-2,s))=\{0\}$.
\end{theorem}

\begin{proof}
Let $g\in M_0(F(p,p\alpha-2,s))$.
Then obviously, $g\in M(F(p,p\alpha-2,s))$.
By Theorem~\ref{20} we know that $g\in H^{\infty}$.
Let $\{a_n\}$ be any sequence in $D$ such that $|a_n|\to1$ as $n\to\infty$.
Let
$$
\psi_n(z)=\frac{(1-|a_n|^2)^2}{(1-\bar a_nz)^{\alpha+1}}-\frac{1-|a_n|^2}{(1-\bar a_nz)^{\alpha}}.
$$
Then, as in the proof of Theorem~\ref{20}, we know
$$
\sup_{n\ge1}\|\psi_n\|_{F(p,p\alpha-2,s)}\le C<\infty,
\qquad
\psi_n(a_n)=0, \qquad \psi_n'(a_n)=\bar a_n(1-|a_n|^2)^{-\alpha}.
$$
Since $M_g$ is compact on $F(p,p\alpha-2,s)$ and $F(p,p\alpha-2,s)\subset B^{\alpha}$,
we get that $M_g$ is compact from $F(p,p\alpha-s,s)$ to $B^{\alpha}$.
Thus
\begin{eqnarray*}
0&\leftarrow&\lim_{n\to\infty}\|M_g\psi_n\|_{B^{\alpha}}\\
&=&\lim_{n\to\infty}\sup_{z\in D}|g'(z)\psi_n(z)+g(z)\psi_n'(z)|(1-|z|^2)^{\alpha}\\
&\ge&\lim_{n\to\infty}|g'(a)\psi_n(a)+g(a)\psi_n'(a)|(1-|a|^2)^{\alpha}\\
&=&\lim_{n\to\infty}|g(a_n)||a_n|(1-|a_n|^2)^{-\alpha}(1-|a_n|^2)^{\alpha}\\
&=&\lim_{n\to\infty}|g(a_n)||a_n|.
\end{eqnarray*}
Since $g\in H^{\infty}$, we must have $g=0$.
The proof is complete.
\end{proof}

\section{An open question}
Finally, we want to mention a natural question that remains open. The question concerns
the embedding $I:\,F(p,p\alpha-2,s)\to T^{\infty}_{p,s}(\mu)$ in the case $\alpha>1$. From
Theorem \ref{14a} we know that $\mu$ being a Carleson measure for $D^p_{s+\alpha p-2}$ is a sufficient condition
for the boundedness. On the other hand, it is easy to see that if the embedding is bounded,
 then $\mu$ must be an $[s+p(\alpha-1)]$-Carleson measure. In the cases $0<p\le 1$; or $s+p(\alpha-1)>1$;
or $1<p\le 2$ and $s+p(\alpha-1)=1$ it is well known that the two conditions are equivalent,
allowing to obtain a complete description in that case (see Theorem \ref{14}). However,
it is known that the two conditions are no longer equivalent in the remaining cases. So, what
 is the criterion for the boundedness and compactness of the embedding
 $I:\,F(p,p\alpha-2,s)\to T^{\infty}_{p,s}(\mu)$ in these cases? Is the converse of Theorem \ref{14a} true?
 How about the boundedness of the Riemann-Stieltjes operator $J_g$ and the
multiplication operator $M_g$ on $F(p,p\alpha-2,s)$ for this case?

\bigskip
\end{document}